\documentclass[12pt]{amsart}
\usepackage{fullpage}
\usepackage{amsfonts,amssymb,amscd,amsmath,enumerate,verbatim,color,xcolor}
\usepackage[latin1]{inputenc}
\usepackage{amscd}
\usepackage{amsthm}
\usepackage{latexsym}
\usepackage{tikz}
\usepackage{graphicx} 

\def\NZQ{\mathbb}               

\def\ZZ{{\NZQ Z}}
\def\RR{{\NZQ R}}

\def\frk{\mathfrak}               

\def\Phi{{\frk N}}
\def\ab{{\mathbf a}}
\def\bb{{\mathbf b}}
\def\cb{{\mathbf c}}
\def\db{{\mathbf d}}
\def\eb{{\mathbf e}}

\def\xb{{\mathbf x}}
\def\yb{{\mathbf y}}

\def\pb{{\mathbf p}}
\def\qb{{\mathbf q}}

\def\opn#1#2{\def#1{\operatorname{#2}}} 
\opn\gr{gr}

\def\Fc{{\mathcal F}}

\def\Pc{{\mathcal P}}

\newtheorem{Theorem}{Theorem}[section]
\newtheorem{Lemma}[Theorem]{Lemma}
\newtheorem{Corollary}[Theorem]{Corollary}
\newtheorem{Proposition}[Theorem]{Proposition}

\theoremstyle{definition}
\newtheorem{Remark}[Theorem]{Remark}

\newtheorem{Example}[Theorem]{Example}

\let\varepsilon\varepsilon
\let\phi=\varphi
\let\kappa=\varkappa

\title{Algebraic aspects of unconditional lattice polytopes}
\author[K. Mori, R. Motomura, H. Ohsugi and A. Tsuchiya]{Kenta Mori, Ryo Motomura, Hidefumi Ohsugi and Akiyoshi Tsuchiya}

\address{Kenta Mori,
	Department of Mathematical Sciences,
	School of Science,
	Kwansei Gakuin University,
	Sanda, Hyogo 669-1330, Japan}
\email{k-mori@kwansei.ac.jp}

\address{Ryo Motomura,
Department of Information Science,
Faculty of Science,
Toho University,
2-2-1 Miyama, Funabashi, Chiba 274-8510, Japan} 
\email{6524012m@st.toho-u.ac.jp}

\address{Hidefumi Ohsugi,
	Department of Mathematical Sciences,
	School of Science,
	Kwansei Gakuin University,
	Sanda, Hyogo 669-1330, Japan} 
\email{ohsugi@kwansei.ac.jp}

\address{Akiyoshi Tsuchiya,
Department of Information Science,
Faculty of Science,
Toho University,
2-2-1 Miyama, Funabashi, Chiba 274-8510, Japan} 
\email{akiyoshi@is.sci.toho-u.ac.jp}

\begin{document}

\begin{abstract}
Unconditional polytopes are convex polytopes that are symmetric with respect to all coordinate hyperplanes and arise naturally from anti-blocking polytopes by reflection. This paper investigates algebraic relations between an anti-blocking lattice polytope and its associated unconditional lattice polytope. We prove that the toric ring of an anti-blocking lattice polytope is normal if and only if the toric ring of the associated unconditional lattice polytope is normal. We also show that the toric ideal of an anti-blocking lattice polytope is generated by quadratic binomials if and only if the same holds for the associated unconditional lattice polytope. As an application, we obtain a graph-theoretic characterization of quadratic generation of symmetric stable set ideals.
\end{abstract}

\maketitle

\section{Introduction}

A \emph{lattice polytope} is a convex polytope all of whose vertices have integer coordinates.
A lattice polytope \(\mathcal{P} \subset \mathbb{R}_{\geq 0}^n\) of dimension \(n\) is called \emph{anti-blocking} if for any \(\mathbf{y}=(y_1,\dots,y_n) \in \mathcal{P}\) and \(\mathbf{x}=(x_1,\dots,x_n) \in \mathbb{R}^n\) with \(0 \leq x_i \leq y_i\) for all \(i\), it holds that \(\mathbf{x} \in \mathcal{P}\).
Anti-blocking polytopes were studied and named by Fulkerson \cite{Fulkerson1,Fulkerson2} in the context of combinatorial optimization, but they are also known as convex corners or down-closed polytopes; see, for example, \cite{downclosed}.
Typical examples include stable set polytopes of graphs and chain polytopes of posets.

For \(\varepsilon = (\varepsilon_1,\dots,\varepsilon_n) \in \{-1,1\}^n\) and \(\mathbf{x}=(x_1,\dots,x_n) \in \mathbb{R}^n\), set
\[
\varepsilon \mathbf{x} := (\varepsilon_1x_1,\dots,\varepsilon_nx_n).
\]
Given an anti-blocking lattice polytope \(\mathcal{P} \subset \mathbb{R}_{\geq 0}^n\) of dimension \(n\), we define
\[
\mathcal{P}^{\pm}
:=
\{
\varepsilon \mathbf{x} \in \mathbb{R}^n :
\varepsilon \in \{-1,1\}^n,\ \mathbf{x} \in \mathcal{P}
\}.
\]
Since \(\mathcal{P}\) is anti-blocking, \(\mathcal{P}^{\pm}\) is a convex lattice polytope.
Moreover, for any \(\varepsilon \in \{-1,1\}^n\) and \(\mathbf{x} \in \mathcal{P}^{\pm}\), one has \(\varepsilon \mathbf{x} \in \mathcal{P}^{\pm}\).
A lattice polytope \(\mathcal{Q} \subset \mathbb{R}^n\) of dimension \(n\) is called \emph{unconditional} if
\(\varepsilon \mathbf{x} \in \mathcal{Q}\) for every \(\varepsilon \in \{-1,1\}^n\) and every \(\mathbf{x} \in \mathcal{Q}\).
Thus \(\mathcal{P}^{\pm}\) is an unconditional lattice polytope naturally associated with \(\mathcal{P}\).
In particular, \(\mathcal{P}^{\pm}\) is symmetric with respect to all coordinate hyperplanes, and hence the origin \(\mathbf{0}\) belongs to the interior of \(\mathcal{P}^{\pm}\).
Unconditional convex bodies form a classical class in convex geometry.
From the viewpoint of lattice polytopes, unconditional polytopes have recently been studied in connection with combinatorics and commutative algebra.
For example, Kohl, Olsen and Sanyal investigated unconditional reflexive polytopes and characterized them in terms of perfect graphs \cite{KOS}.
Moreover, unconditional polytopes are closely related to locally anti-blocking polytopes, whose \(h^*\)-polynomials and \(\gamma\)-positivity have been studied \cite{OhsugiTsuchiya2021locally}.
These developments motivate the study of algebraic properties of unconditional lattice polytopes arising from anti-blocking lattice polytopes.

The main purpose of this paper is to investigate algebraic relations between an anti-blocking lattice polytope and its associated unconditional lattice polytope. More precisely, we study to what extent algebraic properties of the toric rings and toric ideals of these two polytopes are shared by each other. We show that this is indeed the case for normality and quadraticity.

Let \(K[{\bf t}^{\pm1}, s] = K[t_1^{\pm1}, \ldots, t_n^{\pm1}, s]\) be the Laurent polynomial ring in \(n+1\) variables over a field \(K\).
If \(\mathcal{P} \cap \mathbb{Z}^n = \{\mathbf{a}_1,\dots,\mathbf{a}_m\}\), then the \emph{toric ring} of \(\mathcal{P}\) is the subalgebra \(K[\mathcal{P}]\) of \(K[{\bf t}^{\pm1}, s]\) generated by
$\{{\bf t}^{\mathbf{a}_1}s,\dots,{\bf t}^{\mathbf{a}_m}s\}$.
Let \(K[\mathbf{x}] = K[x_1,\dots,x_m]\) be the polynomial ring in \(m\) variables over \(K\).
The \emph{toric ideal} \(I_{\mathcal{P}}\) of \(\mathcal{P}\) is the kernel of the surjective homomorphism
\[
\pi : K[\mathbf{x}] \rightarrow K[\mathcal{P}]
\]
defined by \(\pi(x_i) = {\bf t}^{\mathbf{a}_i}s\) for \(1 \leq i \leq m\).
See, e.g., \cite{HHObook} for details on toric rings and toric ideals.

Our first main result concerns normality of toric rings.

\begin{Theorem}\label{thm:normality}
Let \(\mathcal{P} \subset \mathbb{R}^n\) be an anti-blocking lattice polytope.
Then \(K[\mathcal{P}]\) is normal if and only if \(K[\mathcal{P}^{\pm}]\) is normal.
\end{Theorem}

Our second main result concerns quadratic generation of toric ideals.

\begin{Theorem}\label{thm:quadiff}
Let \(\mathcal{P} \subset \mathbb{R}^n\) be an anti-blocking lattice polytope.
Then \(I_{\mathcal{P}}\) is generated by quadratic binomials if and only if \(I_{\mathcal{P}^{\pm}}\) is generated by quadratic binomials.
\end{Theorem}

As an application, we obtain a graph-theoretic characterization of quadratic generation of symmetric stable set ideals.

The paper is organized as follows.
In Section~\ref{sec:normality}, we introduce locally anti-blocking lattice polytopes and prove a more general normality result in this setting, from which Theorem~\ref{thm:normality} follows immediately.
Section~\ref{sec:quadraticity} is devoted to quadratic generation.
In Subsection~\ref{subsec:ascending}, we prove the implication from \(I_{\Pc}\) to \(I_{\Pc^\pm}\).
In Subsection~\ref{subsec:proofquadraticity}, we prove the converse implication from \(I_{\Pc^\pm}\) to \(I_{\Pc}\).
Finally, in Subsection~\ref{subsec:stable}, we apply Theorem~\ref{thm:quadiff} to symmetric stable set polytopes and obtain a characterization in terms of Kempe equivalence.

\subsection*{Acknowledgment}
This work was supported by 
JSPS KAKENHI 22K13890,\\ 24K00534, 26K00618 and 26K16970.

\section{Normality of the toric rings of locally anti-blocking polytopes}
\label{sec:normality}

In this section, we prove a more general normality result for locally anti-blocking lattice polytopes, from which Theorem~\ref{thm:normality} follows immediately.

We begin by recalling the integer decomposition property.
A lattice polytope \(\mathcal{P}\subset \mathbb{R}^n\) is said to have the \emph{integer decomposition property} (\emph{IDP}, for short) if for any integer \(t\ge 1\) and any lattice point \(\mathbf{x}\in t\mathcal{P}\cap \mathbb{Z}^n\), there exist lattice points
\(\mathbf{x}_1,\dots,\mathbf{x}_t\in \mathcal{P}\cap \mathbb{Z}^n\) such that
\[
\mathbf{x}=\mathbf{x}_1+\cdots+\mathbf{x}_t.
\]
Here $t \Pc := \{t \alpha  :  \alpha \in \Pc\}$.
A lattice polytope which has the integer decomposition property is called \textit{IDP}.
We also recall the notion of spanning.
A lattice polytope \(\mathcal{P}\subset \mathbb{R}^n\) is called \emph{spanning} if the lattice points in \(\mathcal{P}\) affinely generate \(\mathbb{Z}^n\).
In general, normality and IDP are not equivalent.
However, for any spanning lattice polytope $\Pc$, the toric ring \(K[\mathcal{P}]\) is normal if and only if \(\mathcal{P}\) is IDP.
Since every full-dimensional anti-blocking lattice polytope is spanning, this equivalence applies to the polytopes considered in this paper.

We now recall the notion of locally anti-blocking lattice polytopes.
For \(\varepsilon=(\varepsilon_1,\dots,\varepsilon_n)\in \{-1,1\}^n\), let
\[
\mathbb{R}^n_{\varepsilon}
=
\{
(x_1,\dots,x_n)\in \mathbb{R}^n : x_i\varepsilon_i\ge 0 \text{ for all } 1\le i\le n
\}
\]
denote the closed orthant corresponding to \(\varepsilon\).
A lattice polytope \(\mathcal{P}\subset \mathbb{R}^n\) of dimension \(n\) is called \emph{locally anti-blocking} if, for each \(\varepsilon\in \{-1,1\}^n\), there exists an anti-blocking lattice polytope \(\mathcal{P}_\varepsilon\subset \mathbb{R}_{\ge 0}^n\) of dimension \(n\), which is called the \textit{anti-blocking piece} of $\Pc$ corresponding to $\varepsilon$, such that
\[
\mathcal{P}\cap \mathbb{R}^n_\varepsilon
=
\mathcal{P}_\varepsilon^{\pm}\cap \mathbb{R}^n_\varepsilon.
\]
In particular, every unconditional lattice polytope is locally anti-blocking.
This class provides a natural common framework containing both anti-blocking and unconditional lattice polytopes.
For a locally anti-blocking lattice polytope \(\mathcal{P}\), each anti-blocking piece \(\mathcal{P}_\varepsilon\) is full-dimensional, hence spanning.
Therefore, in our setting, the normality statement is reduced to the corresponding statement for IDP.

We will use the following elementary lemma.

\begin{Lemma}\label{lem:union-idp}
Let \(\mathcal{P},\mathcal{P}_1,\dots,\mathcal{P}_r\subset \mathbb{R}^n\) be lattice polytopes such that
\[
\mathcal{P}=\mathcal{P}_1\cup \cdots \cup \mathcal{P}_r.
\]
If each \(\mathcal{P}_i\) is IDP, then \(\mathcal{P}\) is also IDP.
\end{Lemma}

We are now in a position to prove the main result of this section.

\begin{Theorem}\label{thm:normality-local}
Let \(\mathcal{P} \subset \mathbb{R}^n\) be a locally anti-blocking lattice polytope of dimension \(n\).
Then \(K[\mathcal{P}]\) is normal if and only if the toric ring \(K[\mathcal{P}_\varepsilon]\)
of the anti-blocking piece \(\mathcal{P}_\varepsilon\)
is normal for each \(\varepsilon \in \{-1,1\}^n\).
\end{Theorem}

\begin{proof}
Assume first that each \(\mathcal{P}_\varepsilon \) is IDP.
Since each \(\mathcal{P}\cap \mathbb{R}^n_\varepsilon=
\mathcal{P}_\varepsilon^{\pm}\cap \mathbb{R}^n_\varepsilon\) is unimodularly equivalent to \(\mathcal{P}_\varepsilon\), 
\[
\mathcal{P}
=
\bigcup_{\varepsilon\in\{-1,1\}^n}
\left(\mathcal{P}\cap \mathbb{R}^n_\varepsilon\right)
\]
is IDP by Lemma~\ref{lem:union-idp}.

Suppose that some \(\Pc_\varepsilon \subset  \RR_{\ge 0}^n\) is not IDP.
We may assume that \(\varepsilon =(1,\ldots,1)\), and hence
$\Pc_\varepsilon \subset \Pc$.
Then there exists \(\xb =(x_1,\dots,x_n) \in t\Pc_\varepsilon \cap \ZZ^n\) with an integer \(t \geq 2\)
such that there exist no \(\xb_1, \dots ,\xb_t\in \Pc_\varepsilon \cap  \ZZ^n\) with
\[
\xb=\xb_1+\dots +\xb_t.
\]

Suppose that \(\Pc\) is IDP.
Then since \(\xb \in t\Pc_\varepsilon \cap \ZZ^n \subset t\Pc \cap \ZZ^n\), there exist \(t\) lattice points \(\yb_1,\ldots,\yb_t\) in \(\Pc \cap \ZZ^n\) with
\[
\xb = \yb_1+\cdots+\yb_t.
\]
For each \(\yb_i\), we write \(\yb_i=(y_{i1},\ldots,y_{in})\).

If \(y_{ik}>0\) and \(y_{jk}<0\) for some \(1 \leq i , j \leq t\) and some \(1 \leq k \leq n\), then \(\yb_i-\eb_k\) and \(\yb_j+\eb_k\),
where $\eb_k$ is the $k$-th unit coordinate vector, belong to \(\Pc\) since \(\Pc\) is locally anti-blocking.
Then we can replace \(\yb_i\) and \(\yb_j\) with \(\yb_i-\eb_k\) and \(\yb_j+\eb_k\), respectively, since
\[
\yb_i + \yb_j = (\yb_i-\eb_k) +(\yb_j+\eb_k).
\]
By performing this operation repeatedly, we may assume that for any \(1 \leq i < j \leq t\) and any \(1 \leq k \leq n\),
\[
y_{ik} y_{jk} \geq 0.
\]

Since \(\xb \in t \Pc_\varepsilon \subset \RR_{\ge 0}^n\),
\[
0 \le x_k = y_{1k} + \cdots + y_{tk}.
\]
Hence each \(y_{ik}\) is nonnegative.
Therefore each \(\yb_i\) belongs to
\[
\Pc\cap \RR_{\ge0}^n
=
\Pc_\varepsilon^{\pm}\cap \RR_{\ge0}^n
=
\Pc_\varepsilon,
\]
and thus
\[
\yb_1,\ldots,\yb_t \in \Pc_\varepsilon \cap \ZZ^n,
\]
a contradiction.
Thus \(\Pc\) is not IDP.
\end{proof}

Theorem~\ref{thm:normality} follows by applying Theorem~\ref{thm:normality-local} to the unconditional polytope \(\mathcal{P}^{\pm}\), whose anti-blocking pieces are all unimodularly equivalent to \(\mathcal{P}\).

\section{Quadratic generation}\label{sec:quadraticity}
In this section, we prove Theorem~\ref{thm:quadiff}.

\subsection{Ascending quadratic generation from \(I_{\Pc}\) to \(I_{\Pc^\pm}\)}
\label{subsec:ascending}

Two vectors \(\ab=(a_1,\dots,a_n), \bb=(b_1,\dots,b_n)\in \RR^n\) are called \emph{separable} if
\[
a_i b_i <0
\]
for some \(1 \le i \le n\). If $\Pc \subset \RR^n$ is an anti-blocking lattice polytope and $\ab, \bb \in \Pc^{\pm} \cap \ZZ^n$ are separable, then there exist $\varepsilon \in \{-1,1\}^n$  and $\pb, \qb \in \Pc \cap \ZZ^n$ such  that $\varepsilon(\ab+\bb)=\pb+\qb$.

The following proposition may be viewed as a generator version of \cite[Theorem 7.2]{KOS}.

\begin{Proposition}\label{prop:GB}
Let $\Pc \subset \RR^n$ be an anti-blocking lattice polytope.
Assume that $I_{\Pc}$ is generated by 
the set of binomials $\Fc$.
Let $\Fc^{\pm}$ be the set consisting of the binomials 
\[
\prod_{i=1}^rx_{\varepsilon \ab_i}-\prod_{i=1}^r x_{\varepsilon \bb_i}
\]
for any $\prod_{i=1}^rx_{\ab_i}-\prod_{i=1}^r x_{\bb_i} \in \Fc$ and $\varepsilon \in \{-1,1\}^n$, and
\[
x_{\ab}x_{\bb}-x_{\varepsilon \pb}x_{\varepsilon \qb}
\]
for any separable $\ab, \bb \in \Pc^{\pm} \cap \ZZ^n$ and $\varepsilon \in \{-1,1\}^n$ such that $\varepsilon(\ab+\bb)=\pb+\qb$ and $\pb, \qb \in \Pc \cap \ZZ^n$.
Then $I_{{\Pc}^{\pm}}$ is generated by $\Fc^{\pm}$.
\end{Proposition}

\begin{proof}
It is easy to see that $\Fc^{\pm} \subset I_{{\Pc}^{\pm}}$.
Let
\[
f=
\prod_{i=1}^rx_{\varepsilon_i \alpha_i}-\prod_{i=1}^r x_{\varepsilon_i' \beta_i}
\in I_{{\Pc}^{\pm}},
\]
where $\alpha_i, \beta_i \in \Pc \cap \ZZ^n$,
and $\varepsilon_i, \varepsilon_i' \in \{-1,1\}^n$.
We show that $f$ is generated by $\Fc^\pm$.
Suppose that $\varepsilon_{j_1} \alpha_{j_1}, 
\varepsilon_{j_2} \alpha_{j_2}$ are separable
for some $1 \le j_1 < j_2 \le r$.
Let
\[
f' = 
x_{\varepsilon_{j_1} \alpha_{j_1}} x_{\varepsilon_{j_2} \alpha_{j_2}}-x_{\varepsilon \pb}x_{\varepsilon \qb},
\]
where 
$\varepsilon \in \{-1,1\}^n$ such that $\varepsilon(\varepsilon_{j_1} \alpha_{j_1} +\varepsilon_{j_2} \alpha_{j_2})=\pb+\qb$ and $\pb, \qb \in \Pc \cap \ZZ^n$.
Then $f'$ belongs to $\Fc^\pm$, and 
\[
x_{\varepsilon \pb}x_{\varepsilon \qb}
\prod_{i \ne j_1,j_2} x_{\varepsilon_i \alpha_i}-\prod_{i=1}^r x_{\varepsilon_i' \beta_i}
=
f - f'
\prod_{i \ne j_1,j_2} x_{\varepsilon_i \alpha_i}
\in I_{{\Pc}^{\pm}}.
\]
We may consider this binomial instead of $f$.
Repeating the same argument finite times,
we may assume that 
\[
f=
\prod_{i=1}^rx_{\varepsilon \alpha_i}-\prod_{i=1}^r x_{\varepsilon  \beta_i},
\]
for some $\varepsilon \in \{-1,1\}^n$.
Then $\prod_{i=1}^r x_{\alpha_i}-\prod_{i=1}^r x_{\beta_i}$ belongs to $I_\Pc$.
If $\prod_{i=1}^r x_{\alpha_i}-\prod_{i=1}^r x_{\beta_i}$
is generated by 
\[
\prod_{i=1}^{r_1} x_{\ab_i^{(1)}}-\prod_{i=1}^{r_1} x_{\bb_i^{(1)}},
 \ \ \ldots \ \ , \ \ 
\prod_{i=1}^{r_s} x_{\ab_i^{(s)}}-\prod_{i=1}^{r_s} x_{\bb_i^{(s)}}
\in \Fc,\]
then $f$ is generated by
\[
\prod_{i=1}^{r_1} x_{\varepsilon \ab_i^{(1)}}-\prod_{i=1}^{r_1} x_{\varepsilon \bb_i^{(1)}},
 \ \ \ldots \ \ , \ \ 
\prod_{i=1}^{r_s} x_{\varepsilon \ab_i^{(s)}}-\prod_{i=1}^{r_s} x_{\varepsilon \bb_i^{(s)}}
\in \Fc^\pm,\]
as desired.
\end{proof}

As an immediate consequence, we obtain the ``only if'' part of Theorem~\ref{thm:quadiff}.
\begin{Corollary}\label{cor:quad-ascends}
Let \(\Pc \subset \RR^n\) be an anti-blocking lattice polytope.
If \(I_{\Pc}\) is generated by quadratic binomials, then so is \(I_{\Pc^\pm}\).
\end{Corollary}


\subsection{Descending quadratic generation from \(I_{\Pc^\pm}\) to \(I_{\Pc}\)}
\label{subsec:proofquadraticity}

Let $\Pc \subset \RR^n$ be an anti-blocking lattice polytope, and let $A = \mathcal{P} \cap \ZZ^n$.
Let
$K[\xb^\pm]=K[x_{\xi} : \xi \in A^\pm]$
and 
$K[\xb]=K[x_{\ab} : \ab \in A]$
be polynomial rings
over a field $K$.
Then 
\[
I_{\Pc}=I_{\Pc^\pm}\cap K[\xb] \subset I_{\Pc^\pm}
\subset K[\xb^\pm].
\]
For an ordered list
\[
\Lambda=((\ab_1,\varepsilon_1),\ldots,(\ab_r,\varepsilon_r))
,
\]
where $\ab_j = (a_{j,1},\dots,a_{j,n})\in A$ and $\varepsilon_j = (\varepsilon_{j,1},\dots,\varepsilon_{j,n})\in\{-1,1\}^n$, define
\[
\omega(\Lambda)=(\omega_1(\Lambda), \dots, \omega_n(\Lambda)):=\sum_{j=1}^r \varepsilon_j\ab_j\in \ZZ^n.
\]

Assume that
$\omega(\Lambda)$ is nonnegative.
For each $k$-th entry $\omega_k(\Lambda)$ ($ \ge 0$) of $\omega(\Lambda) \in \ZZ_{\ge 0}^n$, define an ordered list 
\[
P_k(\Lambda)
:=
(1,\dots,1, 2,\dots,2, \dots, r, \dots,r),
\]
where each $j \in [r]$ appears
\[
\left\{\begin{array}{cl}
a_{j,k} & \mbox{if } \varepsilon_{j,k}=1,\\
\\
0 &  \mbox{if } \varepsilon_{j,k}=-1
\end{array}
\right.
\]
times in $P_k(\Lambda)$.
%

\begin{Lemma}\label{lem:c_is_welldefined}
With the notation above,
$\omega_k(\Lambda)$ is less than or equal to 
the length of $P_k(\Lambda)$.
\end{Lemma}

\begin{proof}
Let $\ell$ be the length of $P_k(\Lambda)$.
    By construction, we have
\[
\ell = \sum_{j \; : \;  \varepsilon_{j,k}=1} a_{j,k}
\ge
\sum_{j=1}^r \varepsilon_{j,k}a_{j,k}
=
\omega_k(\Lambda).
\]
\end{proof}

For each $j\in[r]$, let $c_{j,k}(\Lambda)$ be the number of times that $j$ appears among the
first $\omega_k(\Lambda)$ entries of $P_k(\Lambda)$.

\begin{Example}\label{exa:welldefined}
If
$
(\varepsilon_{1,k}a_{1,k},\varepsilon_{2,k}a_{2,k},\varepsilon_{3,k}a_{3,k},\varepsilon_{4,k}a_{4,k})
=(+2,-1,+3,0),
$
then
\[P_k(\Lambda)=(1,1,3,3,3).\]
On the other hand, $\omega_k(\Lambda)=4$ and the first four  entries of $P_k(\Lambda)$ is $(1,1,3,3)$, and hence
$(c_{1,k}(\Lambda),c_{2,k}(\Lambda),c_{3,k}(\Lambda),c_{4,k}(\Lambda)) = (2,0,2,0)$.
\end{Example}

Now define
\[
\cb_j(\Lambda):=(c_{j,1}(\Lambda),\ldots,c_{j,n}(\Lambda))
\qquad (1\le j\le r),
\]
and set
\[
\nu(\Lambda):=(\cb_1(\Lambda),\ldots,\cb_r(\Lambda)).
\]

\begin{Lemma}\label{lem:canonical-positive-projection}
With the notation above, one has
\[
\cb_j(\Lambda)\in A
\qquad (1\le j\le r),
\]
and
\[
\sum_{j=1}^r \cb_j(\Lambda)=\omega(\Lambda).
\]
Moreover,
\[
\mathbf 0\leq \cb_j(\Lambda)\leq \ab_j
\qquad\text{coordinatewise}\qquad (1\le j\le r).
\]
\end{Lemma}

\begin{proof}
Fix a coordinate $k\in[n]$.
For each $j\in[r]$, the number $c_{j,k}(\Lambda)$ counts how many copies of $j$ are selected
from $P_k(\Lambda)$.
Therefore
\[
0\leq c_{j,k}(\Lambda)\leq a_{j,k},
\]
i.e.,
\[
\mathbf 0\leq \cb_j(\Lambda)\leq \ab_j
\qquad\text{coordinatewise}.
\]
Since $\Pc$ is anti-blocking and $\ab_j\in \Pc$, we obtain
\[
\cb_j(\Lambda)\in \Pc\cap\ZZ^n=A.
\]
Finally, since the equality
\[
\sum_{j=1}^r c_{j,k}(\Lambda)=\omega_k(\Lambda)
\]
holds for every coordinate $k$, we have
\[
\sum_{j=1}^r \cb_j(\Lambda)=\omega(\Lambda),
\]
as desired.
\end{proof}

For such an ordered list $\Lambda$, 
from Lemma \ref{lem:canonical-positive-projection},
$\nu(\Lambda) =
(\cb_1(\Lambda),\ldots,\cb_r(\Lambda))
$ belongs to $A^r$.
Hence we can define the monomial
\[
m(\nu(\Lambda)):=x_{\cb_1(\Lambda)}\cdots x_{\cb_r(\Lambda)}.
\]

\begin{Remark}\label{rem:nu-fixes-positive-general}
For an ordered list
$\Lambda=((\ab_1,\varepsilon),\ldots,(\ab_r,\varepsilon))$,
where $\ab_j\in A$ and
$\varepsilon =(1,\ldots,1)$,
$
\omega(\Lambda)=\ab_1+\cdots+\ab_r
$
and the construction above yields
$
\cb_j(\Lambda)=\ab_j
\ (1\le j\le r).
$
Hence
$
m(\nu(\Lambda))=x_{\ab_1}\cdots x_{\ab_r}.
$
\end{Remark}

\begin{Lemma}\label{lem:two-slot-compatibility}
Let
\begin{align*}
\Lambda&=((\ab_1,\varepsilon_1),\ldots,(\ab_{r-2},\varepsilon_{r-2}),(\ab,\varepsilon),(\bb,\delta)),\\
\Lambda'&=((\ab_1,\varepsilon_1),\ldots,(\ab_{r-2},\varepsilon_{r-2}),(\ab',\varepsilon'),(\bb',\delta'))
\end{align*}
be ordered lists  with $\ab_j, \ab, \bb, \ab', \bb' \in A$, $\varepsilon_j, \varepsilon, \delta, \varepsilon', \delta' \in\{-1,1\}^n$.
Suppose that
\[
\omega(\Lambda)=\omega(\Lambda')\in \ZZ_{\geq 0}^n.
\]
Then 
$m(\nu(\Lambda)) -m(\nu(\Lambda')) $
is a monomial multiple of a quadratic binomial in $I_\mathcal{P}$.
\end{Lemma}

\begin{proof}
Let
\[
\nu(\Lambda)=(\cb_1,\ldots,\cb_{r-2},\cb,\db),\ \ \ 
\nu(\Lambda')=(\cb_1',\ldots,\cb_{r-2}',\cb',\db').
\]

Fix a coordinate $k\in[n]$.
Consider the ordered lists $P_k(\Lambda)$ and $P_k(\Lambda')$ used in the definition of
$\nu(\Lambda)$ and $\nu(\Lambda')$.
The contributions coming from the first $r-2$ entries of $\Lambda$, $\Lambda'$ are identical in both lists and occur
before the last two entries of $\Lambda$, $\Lambda'$.
Since
$\omega_k(\Lambda)=\omega_k(\Lambda')$,
the number of copies of the index $j$ selected from the first $r-2$ entries is the same in both lists
for every $1 \le j\le r-2$.
Hence
\[
c_{j,k}(\Lambda)=c_{j,k}(\Lambda')
\qquad (1\le j\le r-2).
\]
Since this holds for every coordinate $k$, we obtain
\[
\cb_j=\cb_j'
\qquad (1\le j\le r-2).
\]

Finally, since
\[
\sum_{j=1}^{r-2}\cb_j+\cb+\db
=
\omega(\Lambda)
=
\omega(\Lambda')
=
\sum_{j=1}^{r-2}\cb_j'+\cb'+\db',
\]
we have
$\cb+\db=\cb'+\db'$.
Therefore
$x_{\cb}x_{\db}-x_{\cb'}x_{\db'}\in I_{\Pc}$
and
\[
m(\nu(\Lambda)) -m(\nu(\Lambda')) 
= 
x_{\cb_1} \cdots x_{\cb_{r-2}} (x_{\cb}x_{\db}-x_{\cb'}x_{\db'}) .
\]
\end{proof}

\begin{Lemma}\label{lem:adjacent-swap}
Let
\begin{align*}
\Lambda & = 
((\ab_1,\varepsilon_1),\ldots,(\ab_{p-1},\varepsilon_{p-1}),(\ab_p,\varepsilon_p),(\ab_{p+1},\varepsilon_{p+1}) ,(\ab_{p+2},\varepsilon_{p+2})   ,\ldots,(\ab_r,\varepsilon_r)),\\
\Lambda' & =
((\ab_1,\varepsilon_1),\ldots,(\ab_{p-1},\varepsilon_{p-1}),(\ab_{p+1},\varepsilon_{p+1}),(\ab_{p},\varepsilon_{p}) ,(\ab_{p+2},\varepsilon_{p+2})    ,\ldots,(\ab_r,\varepsilon_r))\end{align*}
be ordered lists with $\ab_j\in A$, $\varepsilon_j\in\{-1,1\}^n$.
Suppose that $\omega(\Lambda) \ (=\omega(\Lambda'))$
is nonnegative.
Then 
$m(\nu(\Lambda))
-
m(\nu(\Lambda'))$ 
is a monomial multiple of a quadratic binomial in $I_\mathcal{P}$.
\end{Lemma}

\begin{proof}
Write
\[
\nu(\Lambda)=(\cb_1,\ldots,\cb_r),
\qquad
\nu(\Lambda')=(\cb_1',\ldots,\cb_r').
\]

We first show that
\[
\cb_j=\cb_j'
\qquad (j\neq p,p+1).
\]
Fix $j$ ($\ne p, p+1$), and fix a coordinate $k\in[n]$.
Recall that $c_{j,k}(\Lambda)$ is the number of copies of the index $j$
among the first $\omega_k(\Lambda)$ entries of $P_k(\Lambda)$, and similarly for
$c_{j,k}(\Lambda')$.

The ordered lists $P_k(\Lambda)$ and $P_k(\Lambda')$ differ only in the relative order of
the block of $p$'s and the block of $(p+1)$'s.
In particular, for every $j\neq p,p+1$, the copies of the index $j$ occur in exactly the same
positions in $P_k(\Lambda)$ and in $P_k(\Lambda')$, i.e.,
\begin{align*}
P_k(\Lambda)
&=
(\underbrace{1,\dots,1}_{m_1}, \underbrace{2,\dots,2}_{m_2}, \dots, \underbrace{r, \dots,r}_{m_r}),\\
P_k(\Lambda')
&=
(\underbrace{1,\dots,1}_{n_1}, \underbrace{2,\dots,2}_{n_2}, \dots, \underbrace{r, \dots,r}_{n_r}),
\end{align*}
where $m_j=n_j \geq 0$ for $j \neq p,p+1$, $m_{p}=n_{p+1} \geq 0$ and $m_{p+1}=n_{p} \geq 0$.
For example, if $p=2$ and 
$P_k(\Lambda)=(1,1,2,2,2,3,5,5)$, then 
$P_k(\Lambda')=(1,1,2,3,3,3,5,5)$.

Since
\[
\omega_k(\Lambda)=\omega_k(\Lambda'),
\]
the cutoff position is the same in both lists.
Therefore the number of copies of $j$ among the first $\omega_k(\Lambda)$ entries of
$P_k(\Lambda)$ is equal to the number of copies of $j$ among the first $\omega_k(\Lambda')$
entries of $P_k(\Lambda')$, that is,
\[
c_{j,k}(\Lambda)=c_{j,k}(\Lambda')
\qquad (j\neq p,p+1).
\]
Since this holds for every coordinate $k$, we obtain
\[
\cb_j=\cb_j'
\qquad (j\neq p,p+1).
\]

Now
\[
\sum_{j=1}^r \cb_j
=
\omega(\Lambda)
=
\omega(\Lambda')
=
\sum_{j=1}^r \cb_j'.
\]
Since $\cb_j=\cb_j'$ for all $j\neq p,p+1$, it follows that
\[
\cb_p+\cb_{p+1}=\cb_p'+\cb_{p+1}'.
\]
Hence
\[
x_{\cb_p}x_{\cb_{p+1}}-x_{\cb_p'}x_{\cb_{p+1}'}\in I_{\Pc}.
\]

Finally,
\[
m(\nu(\Lambda))
=
\left(\prod_{j\neq p,p+1} x_{\cb_j}\right)x_{\cb_p}x_{\cb_{p+1}},
\qquad
m(\nu(\Lambda'))
=
\left(\prod_{j\neq p,p+1} x_{\cb_j}\right)x_{\cb_p'}x_{\cb_{p+1}'},
\]
and hence
\[
m(\nu(\Lambda))-m(\nu(\Lambda'))
=
\left(\prod_{j\neq p,p+1} x_{\cb_j}\right)
\bigl(x_{\cb_p}x_{\cb_{p+1}}-x_{\cb_p'}x_{\cb_{p+1}'}\bigr).
\]
Therefore
\[
m(\nu(\Lambda))-m(\nu(\Lambda'))
\]
is a monomial multiple of a quadratic binomial in $I_{\Pc}$.
\end{proof}

\begin{Corollary}\label{cor:ordering-independence}
Let $\Lambda$ and $\Lambda'$ be two ordered lists
\begin{align*}
\Lambda&=((\ab_1,\varepsilon_1),\ldots,(\ab_{r},\varepsilon_{r})),\\
\Lambda'&=((\ab_1',\varepsilon_1'),\ldots,(\ab_{r}',\varepsilon_{r}'))
\end{align*}
with $\ab_j , \ab_j'\in A$ and $\varepsilon_j, \varepsilon_j'\in\{-1,1\}^n$
such that
\[
x_{\varepsilon_1 \ab_1} \cdots  x_{\varepsilon_r \ab_r}
=
x_{\varepsilon_1' \ab_1'} \cdots  x_{\varepsilon_r' \ab_r'}.
\]
Suppose that
$
\omega(\Lambda) \ (=\omega(\Lambda') ) 
$ is nonnegative.
Then 
the binomial $m(\nu(\Lambda)) - m(\nu(\Lambda'))$
%
in $K[\xb]$
is generated by quadratic binomials in $I_{\Pc}$.
\end{Corollary}

\begin{proof}
Set $\varepsilon_i=(\varepsilon_{i1},\ldots,\varepsilon_{in})$ and $\ab_i=(a_{i1},\ldots,a_{in})$.
We may suppose that for $\Lambda$,
\(\varepsilon_{ij}=1\) whenever \(a_{ij}=0\), and similarly for
\(\Lambda'\). 
Indeed, changing \(\varepsilon_{ij}\) when \(a_{ij}=0\) does not change the lattice point
\(\varepsilon_i\ab_i\), nor the ordered lists \(P_k(\Lambda)\), and hence does not change
\(\omega(\Lambda)\), \(\nu(\Lambda)\), or \(m(\nu(\Lambda))\).
With this convention, the representation of each lattice point
\(\varepsilon_i\ab_i\in A^\pm\) is unique. Hence the assumption
\[
x_{\varepsilon_1 \ab_1} \cdots  x_{\varepsilon_r \ab_r}
=
x_{\varepsilon_1' \ab_1'} \cdots  x_{\varepsilon_r' \ab_r'}
\]
implies that \(\Lambda'\) is obtained from \(\Lambda\) by a permutation of the entries.
Since any permutation is a product of adjacent transpositions, the assertion follows from
Lemma~\ref{lem:adjacent-swap}.
\end{proof}

We are now in the position to show the ``if" part of Theorem \ref{thm:quadiff}.
\begin{Proposition}\label{prop:quadratic-descends}
Let \(\Pc \subset \RR^n\) be an anti-blocking lattice polytope.
If \(I_{\Pc^\pm}\) is generated by quadratic binomials, then  so is \(I_{\Pc}\).
\end{Proposition}


\begin{proof}
Suppose that $I_{\Pc^{\pm}}$ is generated by quadratic binomials.
Let
\[
u-v = x_{\alpha_1} \cdots x_{\alpha_r} - x_{\beta_1} \cdots x_{\beta_r} \in I_{\Pc}
\]
be a homogeneous binomial of degree $r$.
Let 
\begin{align*}
\Lambda_u &= ( (\alpha_1, \varepsilon),\dots,(\alpha_r, \varepsilon) ),     \\
\Lambda_v &= ( (\beta_1, \varepsilon),\dots,(\beta_r, \varepsilon) ),  
\end{align*}
where $\varepsilon = (1,\dots,1)$.
Since $u-v$ belongs to $I_\mathcal{P} \subset K[\xb]$,
it follows that
\[
\omega(\Lambda_u) = \alpha_1 + \dots +\alpha_r
= \beta_1 + \dots + \beta_r = \omega(\Lambda_v)
\in \ZZ_{\ge 0}^n.
\]
%
By Remark~\ref{rem:nu-fixes-positive-general}, we have
\begin{equation}\label{eq:shoki}
m(\nu(\Lambda_u))=u,
\qquad
m(\nu(\Lambda_v))=v.   
\end{equation}


Since
$u-v\in I_{\Pc^\pm}$
and since $I_{\Pc^{\pm}}$ is generated by quadratic binomials, there exists a sequence of monomials
\[
u=M_0,M_1,\ldots,M_t=v
\]
of degree $r$
such that for each $1\leq k\leq t$, the difference
$M_{k-1}-M_k$
is a monomial multiple of a quadratic binomial in $I_{\Pc^\pm}$, i.e., 
\[
M_{k-1}-M_k
=
N_k
(x_{\varepsilon_k\ab_k}x_{\delta_k\bb_k} -  x_{\varepsilon_k'\ab_k'}x_{\delta_k'\bb_k'})
=
N_k \, x_{\varepsilon_k\ab_k}x_{\delta_k\bb_k} - N_k \, x_{\varepsilon_k'\ab_k'}x_{\delta_k'\bb_k'}
\]
for some monomial 
\[
N_k=x_{\varepsilon_1^{(k)} \gamma_1^{(k)}} \cdots x_{\varepsilon_{r-2}^{(k)} \gamma_{r-2}^{(k)}} \in K[\xb^{\pm}]
\]
with
$ \gamma_j^{(k)} \in A$ and $\varepsilon_j^{(k)} \in \{-1,1\}^n$
and some quadratic binomial 
\[
x_{\varepsilon_k\ab_k}x_{\delta_k\bb_k} -  x_{\varepsilon_k'\ab_k'}x_{\delta_k'\bb_k'} \in I_{\Pc^\pm}
\]
with
$\ab_k,\bb_k,\ab_k',\bb_k'\in A$ and $
\varepsilon_k,\delta_k,\varepsilon_k',\delta_k'\in\{-1,1\}^n.
$
%
%
Let
\[
\Gamma_k:= ( (\gamma_1^{(k)},  \varepsilon_1^{(k)} ),
\dots,(\gamma_{r-2}^{(k)},  \varepsilon_{r-2}^{(k)} ))
\]
be the list associated with $N_k$.
Define
\[
\Lambda_{k-1}^{(k)}:=(\Gamma_k,(\ab_k,\varepsilon_k),(\bb_k,\delta_k)),
\qquad
\Lambda_k^{(k)}:=(\Gamma_k,(\ab_k',\varepsilon_k'),(\bb_k',\delta_k')).
\]
We may assume that $\Lambda_u = \Lambda_0^{(1)}$
and $\Lambda_v = \Lambda_t^{(t)}$.
Then
\[
\omega(\Lambda_{k-1}^{(k)})=\omega(\Lambda_k^{(k)})=\omega(\Lambda_u) 
\in \ZZ_{\ge 0}^n
\]
for all $k$.
Set
\[
P_{k-1}^{(k)}:=m(\nu(\Lambda_{k-1}^{(k)})),
\qquad
P_k^{(k)}:=m(\nu(\Lambda_k^{(k)})).
\]
%
From \eqref{eq:shoki}, we have
$P_0^{(1)}=u$ and 
$
P_t^{(t)}=v.
$
By Lemma~\ref{lem:two-slot-compatibility}, the difference
$P_{k-1}^{(k)}-P_k^{(k)}$
is a monomial multiple of a quadratic binomial in $I_{\Pc}$.
Now fix $k$ with $1\leq k\leq t-1$.
The two monomials
$P_k^{(k)}$ and $P_k^{(k+1)}$
come from two ordered lists
$\Lambda_k^{(k)}$
and
$\Lambda_k^{(k+1)}$
of the same monomial 
\[
M_k =
\prod_{(\ab,\varepsilon) \in \Lambda_k^{(k)}} x_{\varepsilon \ab}
=
\prod_{(\ab',\varepsilon') \in \Lambda_k^{(k+1)}} x_{\varepsilon' \ab'}.
\]
Hence, by Corollary~\ref{cor:ordering-independence}, 
$P_k^{(k)}-P_k^{(k+1)} \in K[\xb]$
is generated by
quadratic binomials in $I_{\Pc}$.

We now concatenate the following chains:
\[
P_0^{(1)} \rightsquigarrow P_1^{(1)},
P_1^{(1)} \rightsquigarrow P_1^{(2)},
P_1^{(2)} \rightsquigarrow P_2^{(2)},
P_2^{(2)} \rightsquigarrow P_2^{(3)},
\dots,
P_{t-1}^{(t-1)} \rightsquigarrow P_{t-1}^{(t)},
P_{t-1}^{(t)} \rightsquigarrow P_t^{(t)}.
\]
Here, for each chain of the form
$P_{k-1}^{(k)} \rightsquigarrow P_k^{(k)}$,
$P_{k-1}^{(k)}- P_k^{(k)}$
is a monomial multiple of a quadratic binomial in $I_\mathcal{P}$
by Lemma~\ref{lem:two-slot-compatibility}, and for each chain of the form
$P_k^{(k)} \rightsquigarrow P_k^{(k+1)}$,
$P_k^{(k)} - P_k^{(k+1)}$
is generated by quadratic binomials in $I_\mathcal{P}$
by Corollary~\ref{cor:ordering-independence}.
Thus we obtain a sequence of monomials in $K[\xb]$ from
$u=P_0^{(1)}$
to
$v=P_t^{(t)}$
such that every consecutive difference is 
generated by quadratic binomials 
in $I_{\Pc}$.
Hence  $u-v$ belongs to the ideal generated by
quadratic binomials in $I_{\Pc}$.
Therefore, $I_{\Pc}$ is generated by quadratic binomials.
\end{proof}

Theorem~\ref{thm:quadiff} now follows from Corollary~\ref{cor:quad-ascends} and Proposition~\ref{prop:quadratic-descends}.

\subsection{Application to symmetric stable set polytopes}
\label{subsec:stable}

In this subsection, we apply Theorem~\ref{thm:quadiff} to symmetric stable set polytopes.

Let $G$ be a graph on the vertex set $[n]:=\{1,2,\ldots,n\}$ with the edge set $E(G)$.
Given a subset $S \subset [n]$,
let $G[S]$ denote the induced subgraph of $G$ on the vertex set $S$.
A subset $S \subset [n]$ is called a {\em stable set} (or an {\em independent set}) of $G$
if $\{i,j\} \notin E(G)$ for all $i,j \in S$ with $i \neq j$.
Namely, a subset $S \subset [n]$ is stable if and only if $G[S]$ is an empty graph.
In particular, the empty set $\emptyset$ and any singleton $\{i\}$ with $i \in [n]$
are stable.
Denote $S(G)=\{S_1,\ldots,S_m\}$ the set of all stable sets of $G$.
Given a subset $S \subset [n]$, we associate the $(0,1)$-vector $\rho(S)=\sum_{j \in S} \eb_j \in \RR^n$. 
Note that $\rho(\emptyset)=(0,\ldots,0) \in \RR^n$.
The \textit{stable set polytope} of $G$ is
\[
\Pc_G:={\rm conv} \{ \rho(S_1),\ldots,\rho(S_m)\} \subset \RR^n.
\]
The stable set polytope $\Pc_G$ is anti-blocking of dimension $n$.

We call the toric ring $K[G]:=K[\Pc_G]$ the \textit{stable set ring} of $G$ and the toric ideal $I_{G}:=I_{\Pc_G}$ the \textit{stable set ideal} of $G$.
In \cite{OhsugiTsuchiyaquad}, the authors characterized when $I_G$ is generated by quadratic binomials by using Kempe equivalence.
We also recall the notion of Kempe equivalence.
A \emph{\(k\)-coloring} \(f\) of a graph \(G\) is a map from \([n]\) to \([k]\) such that
\(f(i) \neq f(j)\) for all \(\{i,j\} \in E(G)\).
The \emph{chromatic number} of \(G\), denoted by \(\chi(G)\), is the smallest integer \(k\) for which \(G\) admits a \(k\)-coloring.
Given a \(k\)-coloring \(f\) of \(G\), and \(1 \le i < j \le k\),
let \(H\) be a connected component of the induced subgraph of \(G\)
on the vertex set \(f^{-1}(i) \cup f^{-1}(j)\).
Then we obtain a new \(k\)-coloring \(g\) of \(G\) by setting
\[
g(x) =
\begin{cases}
    f(x) & \text{if } x \notin H,\\
    i & \text{if } x \in H \text{ and } f(x)=j,\\
    j & \text{if } x \in H \text{ and } f(x)=i.
\end{cases}
\]
We say that \(g\) is obtained from \(f\) by a \emph{Kempe switching}.
Two \(k\)-colorings \(f\) and \(g\) are called \emph{Kempe equivalent}, denoted by \(f \sim_k g\), if there exists a sequence \(f_0,f_1,\ldots,f_s\) of \(k\)-colorings of \(G\) such that
$f_0=f,f_s=g$,
and \(f_i\) is obtained from \(f_{i-1}\) by a Kempe switching for each \(i\).

Given a graph $G$ on the vertex set $[n]$, and 
$\ab = (a_1,\ldots,a_n) \in \ZZ_{\ge 0}^n$,
let $G_\ab$ be the graph obtained from $G$ by replacing each vertex $i \in [n]$ 
with a complete graph $G^{(i)}$ of $a_i$ vertices (if $a_i =0$, then just delete the vertex $i$),
and joining all vertices $x \in G^{(i)}$ and $y \in G^{(j)}$ such that $\{i,j\}$ is an edge of $G$.
In particular, if $\ab =(1,\ldots,1)$, then $G_\ab = G$.
If $\ab = {\bf 0}$, then $G_\ab$ is the null graph (a graph without vertices). 
In addition, if $\ab$ is a $(0,1)$-vector, then $G_\ab$ is an induced subgraph of $G$.
If $\ab$ is a positive vector, then $G_\ab$ is called a 
\textit{replication graph} of $G$.
In general, $G_\ab$ is a replication graph of an induced subgraph of $G$.

The following is a characterization of quadratic generation of stable set ideals by using Kempe equivalence.

\begin{Proposition}[{\cite[Theorem~1.1]{OhsugiTsuchiyaquad}}]
\label{prop:stable_quad}
Let \(G\) be a graph.
Then \(I_G\) is generated by quadratic binomials if and only if for every \(\ab \in \ZZ_{\ge 0}^n\) and $k \geq \chi(G_{\ab})$, all \(k\)-colorings of \(G_\ab\) are Kempe equivalent.
\end{Proposition}

For a graph $G$ on $[n]$, we call the unconditional lattice polytope $\Pc_G^{\pm}$ the \textit{symmetric stable set polytope} of $G$. Additionally, we call the toric ring $K[G^{\pm}]:=K[\Pc^{\pm}_G]$ the \textit{symmetric stable set ring} of $G$ and the toric ideal $I_{G^{\pm}}:=I_{\Pc^{\pm}_G}$ the \textit{symmetric stable set ideal} of $G$.

Combining Proposition~\ref{prop:stable_quad} with Theorem~\ref{thm:quadiff}, we immediately obtain Theorem~\ref{thm:app}.

\begin{Theorem}\label{thm:app}
Let $G$ be a graph on $[n]$. Then the following conditions are equivalent{\rm :}
\begin{enumerate}
    \item[{\rm (i)}] $I_{G^{\pm}}$ is generated by quadratic binomials{\rm ;}
    \item[{\rm (ii)}] $I_G$ is generated by quadratic binomials{\rm ;}
    \item[{\rm (iii)}] for every $\ab \in \ZZ_{\ge 0}^n$ and $k \geq \chi(G_{\ab})$,
    all $k$-colorings of $G_\ab$ are Kempe equivalent.
\end{enumerate}
\end{Theorem}

\bibliographystyle{plain}
\bibliography{bibliography}
\end{document}